\documentclass[12pt]{article}
\usepackage{amssymb}
\usepackage{amsthm}
\usepackage{amsmath}
\usepackage{graphicx}
\usepackage{enumerate}

\DeclareMathOperator{\Fil}{Fil}

\newtheorem{theorem}{Theorem}[section]
\newtheorem{definition}[theorem]{Definition}
\newtheorem{lemma}[theorem]{Lemma}
\newtheorem{proposition}[theorem]{Proposition}
\newtheorem{remark}[theorem]{Remark}
\newtheorem{example}[theorem]{Example}
\newtheorem{corollary}[theorem]{Corollary}
\title{Monotone and cone preserving mappings on posets}
\author{Ivan~Chajda and Helmut~L\"anger}
\date{}
\begin{document}

\footnotetext{Support of the research by the Austrian Science Fund (FWF), project I~4579-N, and the Czech Science Foundation (GA\v CR), project 20-09869L, entitled ``The many facets of orthomodularity'', as well as by \"OAD, project CZ~02/2019, entitled ``Function algebras and ordered structures related to logic and data fusion'', and, concerning the first author, by IGA, project P\v rF~2020~014, is gratefully acknowledged.}

\maketitle

\begin{abstract}
We define several sorts of mappings on a poset like monotone, strictly monotone, upper cone preserving and variants of these. Our aim is to characterize posets in which some of these mappings coincide. We define special mappings determined by two elements and investigate when these are strictly monotone or upper cone preserving. If the considered poset is a semilattice then its monotone mappings coincide with semilattice homomorphisms if and only if the poset is a chain. Similarly, we study posets which need not be semilattices but whose upper cones have a minimal element. We extend this investigation to posets that are direct products of chains or an ordinal sum of an antichain and a finite chain. We characterize equivalence relations induced by strongly monotone mappings and show that the quotient set of a poset by such an equivalence relation is a poset again.
\end{abstract}



\section{Introduction}

Partially ordered sets, shortly posets, are relational structures which occur frequently both in various areas of mathematics and in applications. Posets were studied from numerous points of view depending on their application. One possible approach is to consider various mappings on a given poset and check when they coincide. Examples of such mappings are monotone mappings, cone preserving mappings, filter preserving mappings, etc. If the poset in question is of a particular form, e.g.\ if it is a semilattice or lattice, we can consider also homomorphisms. If a poset is directed then it can be converted into a so-called directoid, i.e.\ a groupoid with one binary operation. Homomorphisms of such directed posets were already investigated by the first author in \cite C. For a bit more general relational structures, so-called quasiordered sets, cone preserving mappings were studied in \cite{CH}. Homomorphisms of semilattices were investigated by L.~R.~Berrone in \cite B.

Based on the mentioned results, we introduce a list of interesting mappings on posets and find out how the fact that some of the mappings from this list coincide or satisfy some special assumptions influences the structure of the poset.

We do not consider the research on this topic to be finished. We rather consider our paper as a starting point which could inspire other authors to go on in this direction. We are convinced that the algebraic theory of posets is of a fundamental importance in the whole of mathematics.

\section{Elementary concepts and results}

Concerning the concepts used here, numerous of them are familiarly known and the remaining ones are introduced or recalled below.

Let $\mathbf P:=(P,\leq)$ be a poset, $A,B\subseteq P$ and $a,b\in P$. Then $A\leq B$ should mean $x\leq y$ for all $(x,y)\in A\times B$. Instead of $A\leq\{b\}$, $\{a\}\leq B$ and $\{a\}\leq\{b\}$ we simply write $A\leq b$, $a\leq B$ and $a\leq b$, respectively. The sets
\begin{align*}
L(A) & :=\{x\in P\mid x\leq A\}, \\
U(A) & :=\{x\in P\mid A\leq x\}
\end{align*}
are called the {\em lower} and {\em upper cone} of $A$, respectively. Instead of $L(A\cup B)$, $L(A\cup\{b\})$, $L(\{a\}\cup B)$, $L(\{a,b\})$ and $L(\{a\})$ we simply write $L(A,B)$, $L(A,b)$, $L(a,B)$, $L(a,b)$ and $L(a)$, respectively. In a similar way we proceed for $U$ and in analogous cases. Moreover, put $L^*(a):=(L(a))\setminus\{a\}$ and $U^*(a):=(U(a))\setminus\{a\}$. $\mathbf P$ is called {\em up-directed} if $U(x,y)\neq\emptyset$ for all $x,y\in P$. The subset $A$ of $P$ is called a {\em filter} of $\mathbf P$ if $x\in A$ and $x\leq y$ imply $y\in A$. Let $\Fil\mathbf P$ denote the set of all filters of $\mathbf P$. For each $a\in P$, the set $[a):=\{x\in P\mid a\leq x  \}$ is a filter of $\mathbf P$, the so-called {\em principal filter} generated by $a$.
 
\begin{remark}
If $\mathbf P$ is a poset then $(\Fil\mathbf P,\subseteq)$ is a complete lattice with smallest element $\emptyset$ and greatest element $P$ and
\begin{align*}
  \bigvee_{i\in I}F_i & =\bigcup_{i\in I}F_i, \\
\bigwedge_{i\in I}F_i & =\bigcap_{i\in I}F_i
\end{align*}
for every family $F_i,i\in I,$ of filters of $\mathbf P$.
\end{remark}

A mapping $f:P\rightarrow P$ is called
\begin{enumerate}[(i)]
\item {\em monotone} if $x\leq y$ implies $f(x)\leq f(y)$,
\item {\em strictly monotone} if $x<y$ implies $f(x)<f(y)$,
\item {\em upper cone preserving} if $f(U(x,y))=U(f(x),f(y))$ for all $x,y\in P$,
\item {\em strictly upper cone preserving} if $f(U(x,y))=U(f(x),f(y))$ for all $x,y\in P$ with $x\neq y$,
\item {\rm strongly upper cone preserving} if $f(U(x,y))=U(f(x),f(y))$ for all $x,y\in P$ with $f(x)\neq f(y)$.
\end{enumerate}
Observe that for monotone $f$ we have $f(L(x,y))\subseteq L(f(x),f(y))$ and $f(U(x,y))\subseteq U(f(x),f(y))$ for all $x,y\in P$.

Throughout the paper, we consider only non-void posets.

In the following, for every poset $(P,\leq)$ and every element $a\in P$ let $f_a$ denote the constant mapping from $P$ to $P$ with value $a$.

Using of the mapping $f_a$ which is evidently monotone, we can characterize up-directed posets having a maximal element as follows.
 
\begin{lemma}\label{lem1}
Let $\mathbf P=(P,\leq)$ be a poset and $a\in P$. Then $\mathbf P$ is up-directed and $a$ is a maximal element of $\mathbf P$ if and only if $f_a$ is upper cone preserving.
\end{lemma}

\begin{proof}
Let $b,c\in P$. If $\mathbf P$ is up-directed and $a$ maximal then
\[
f_a(U(b,c))=\{a\}=U(a)=U(a,a)=U(f_a(b),f_a(c))
\]
showing that $f_a$ is upper cone preserving. Conversely, if $f_a$ is upper cone preserving then
\begin{align*}
f_a(U(b,c)) & =U(f_a(b),f_a(c))=U(a,a)=U(a)\supseteq\{a\}\neq\emptyset, \\ 
       U(a) & =U(a,a))=U(f_a(a),f_a(a))=f_a(U(a,a))=\{a\}
\end{align*}
showing that $\mathbf P$ is up-directed and that $a$ is maximal.
\end{proof}

Several elementary facts on cone preserving mappings are stated in the next lemma.

\begin{lemma}\label{lem2}
Let $\mathbf P=(P,\leq)$ be a poset, $f\colon P\rightarrow P$ and $A\subseteq P$. Then the following hold:
\begin{enumerate}[{\rm(i)}]
\item $f$ is monotone if and only if $f(U(x))\subseteq U(f(x))$ for all $x\in P$,
\item if $f$ is upper cone preserving then it is monotone and $f(F)\in\Fil\mathbf P$ for all $F\in\Fil\mathbf P$, 
\item if every monotone mapping from $P$ to $P$ is upper cone preserving then $|P|=1$,
\item if $f$ is monotone then $f(L(A))\subseteq L(f(A))$ and $f(U(A))\subseteq U(f(A))$.
\end{enumerate}
\end{lemma}

\begin{proof}
\
\begin{enumerate}[(i)]
\item This is obvious.
\item Assume $f$ to be upper cone preserving. Then
\[
f(U(x))=f(U(x,x))=U(f(x),f(x))=U(f(x))\text{ for all }x\in P
\]
and hence $f$ is monotone according to (i). Moreover, if $F\in\Fil\mathbf P$ then
\begin{align*}
f(F) & =f(\bigcup_{x\in F}U(x))=\bigcup_{x\in F}f(U(x))=\bigcup_{x\in F}f(U(x,x))=\bigcup_{x\in F}U(f(x),f(x))= \\
     & =\bigcup_{x\in F}U(f(x))\in\Fil\mathbf P.
\end{align*}
\item This follows from Lemma~\ref{lem1} by observing that every constant mapping is monotone.
\item If $f$ is monotone and $a\in f(L(A))$ then there exists some $b\in L(A)$ with $f(b)=a$ and since $f$ is monotone we have $a=f(b)\in L(f(A))$. The statement for $U$ follows by duality.
\end{enumerate}
\end{proof}

\begin{example}
Consider the poset depicted in Figure~1:

\vspace*{-2mm}

\begin{center}
\setlength{\unitlength}{7mm}
\begin{picture}(4,6)
\put(1,1){\circle*{.3}}
\put(3,1){\circle*{.3}}
\put(1,3){\circle*{.3}}
\put(3,3){\circle*{.3}}
\put(2,5){\circle*{.3}}
\put(1,1){\line(0,1)2}
\put(1,1){\line(1,1)2}
\put(3,1){\line(-1,1)2}
\put(3,1){\line(0,1)2}
\put(2,5){\line(-1,-2)1}
\put(2,5){\line(1,-2)1}
\put(.3,.85){$a$}
\put(3.4,.85){$b$}
\put(.3,2.85){$c$}
\put(3.4,2.85){$d$}
\put(1.85,5.4){$1$}
\put(1.2,-.3){{\rm Fig.\ 1}}
\end{picture}
\end{center}

Then $f\colon P\rightarrow P$ defined by
\[
f(x):=\left\{
\begin{array}{ll}
c & \text{if }x\in\{a,b,c\}, \\
1 & \text{otherwise}
\end{array}
\right.
\]
is upper cone preserving and hence monotone according to Lemma~\ref{lem2} {\rm(ii)}. Since the above poset is not a singleton there must exist some monotone mapping which is not upper cone preserving according to Lemma~\ref{lem2} {\rm(iii)}. The mapping $g\colon P\rightarrow P$ defined by
\[
g(x):=\left\{
\begin{array}{ll}
b & \text{if }x=a, \\
x & \text{otherwise}
\end{array}
\right.
\]
is monotone, but not upper cone preserving since
\[
g(U(a,b))=g(\{c,d,1\})=\{c,d,1\}\neq\{b,c,d,1\}=U(b)=U(b,b)=U(g(a),g(b)).
\]
\end{example}

For injective mappings, we can show that they are upper cone preserving provided they preserve principal filters.

\begin{proposition}
Let $(P,\leq)$ be a poset. Then every injective mapping $f:P\rightarrow P$ satisfying $f([x))=[f(x))$ for all $x\in P$ is upper cone preserving.
\end{proposition}

\begin{proof}
If $a,b\in P$ and $f:P\rightarrow P$ is injective and satisfies $f([x))=[f(x))$ for all $x\in P$ then $f(U(a))=U(f(a))$ and
\[
f(U(a,b))=f(U(a)\cap U(b))=f(U(a))\cap f(U(b))=U(f(a))\cap U(f(b))=U(f(a),f(b)).
\]
\end{proof}

\section{Mappings determined by two elements}

In the following, for every poset $(P,\leq)$ and every $a,b\in P$ with $a\neq b$ let $f_{ab}$ denote the mapping from $P$ to $P$ defined by
\[
f_{ab}(x):=\left\{
\begin{array}{ll}
b & \text{if }x=a, \\
x & \text{otherwise}.
\end{array}
\right.
\]

The question when the mapping $f_{ab}$ is strictly monotone is answered in the next proposition.

\begin{proposition}\label{prop1} 
Let $(P,\leq)$ be a poset and $a,b\in P$ with $a\neq b$. Then $f_{ab}$ is strictly monotone if and only if $a\parallel b$, $L^*(a)\subseteq L^*(b)$ and $U^*(a)\subseteq U^*(b)$.
\end{proposition}

\begin{proof}
Obviously, $f_{ab}$ is strictly monotone if $a\parallel b$ (since $b<b$ is impossible) and for all $x\in P$ the following hold:
\begin{eqnarray}
& & a<x\text{ implies }b<x,\label{equ5} \\
& & x<a\text{ implies }x<b.\label{equ7}
\end{eqnarray}
Now (\ref{equ5}) and (\ref{equ7}) are equivalent to $U^*(a)\subseteq U^*(b)$ and $L^*(a)\subseteq L^*(b)$, respectively.
\end{proof}

Similarly, we can ask when the mapping $f_{ab}$ is upper cone preserving. The answer is as follows.

\begin{theorem}\label{th2}
Let $\mathbf P=(P,\leq)$ be a poset and $a,b\in P$ with $a\neq b$. Then $f_{ab}$ is upper cone preserving if and only if $a$ is a minimal element of $\mathbf P$ and $U^*(a)=U(b)$.
\end{theorem}

\begin{proof}
Let $c,d\in P\setminus\{a\}$. First assume $f_{ab}$ to be upper cone preserving. Then $b\leq a$ would imply
\[
a\in U(b)=U(b,b)=U(f_{ab}(a),f_{ab}(a))=f_{ab}(U(a,a))=f_{ab}(U(a))=U^*(a)\cup\{b\}
\]
and hence $a=b$, a contradiction. Therefore $b\not\leq a$. Now
\[
b\in U(b)=U(b,b)=U(f_{ab}(a),f_{ab}(b))=f_{ab}(U(a,b))=U(a,b)\subseteq U(a)
\]
since $a\notin U(a,b)$ and hence $a\leq b$, i.e.\ $a<b$. Now $c<a$ would imply
\[
a\in U(c)=U(c,c)=U(f_{ab}(c),f_{ab}(c))=f_{ab}(U(c,c))=f_{ab}(U(c))=((U(c))\setminus\{a\})\cup\{b\}
\]
and hence $a=b$, a contradiction. This shows that $a$ is minimal. Moreover,
\[
U^*(a)=f_{ab}(U(a))=f_{ab}(U(a,a))=U(f_{ab}(a),f_{ab}(a))=U(b,b)=U(b)
\]
since $b\in U^*(a)$. Conversely, assume $a$ to be minimal and $U^*(a)=U(b)$. Then $a<b$ and
\begin{align*}
f_{ab}(U(a,a)) & =f_{ab}(U(a))=U^*(a)=U(b)=U(b,b)=U(f_{ab}(a),f_{ab}(a)), \\
f_{ab}(U(a,c)) & =U(a,c)=U(b,c)=U(f_{ab}(a),f_{ab}(c)), \\
f_{ab}(U(c,d)) & =U(c,d)=U(f_{ab}(c),f_{ab}(d))
\end{align*}
and hence $f_{ab}$ is upper cone preserving. Observe that $c,d\not\leq a$ because of $c,d\neq a$ and the minimality of $a$. 
\end{proof}

It should be remarked that $U^*(a)=U(b)$ implies $a\prec b$. Namely, from $b\in U(b)=U^*(a)$ we conclude $a<b$. If there would exist some $c\in P$ with $a<c<b$ then $c\in U^*(a)=U(b)$, a contradiction. This shows $a\prec b$.

By Lemma~\ref{lem2} (ii), every upper cone preserving mapping is monotone. The question is for which posets not every strictly monotone mapping is upper cone preserving. The answer is as follows.

\begin{remark}
If $(P,\leq)$ is a poset containing two elements $a$ and $b$ with $a\parallel b$ satisfying $L^*(a)\subseteq L^*(b)$ and $U^*(a)\subseteq U^*(b)$ then not every strictly monotone mapping from $P$ to $P$ is upper cone preserving. Such a poset is depicted in Fig.~1.
\end{remark}

\begin{proof}
Let $(P,\leq)$ be a poset having two elements $a$ and $b$ with $a\parallel b$ satisfying $L^*(a)\subseteq L^*(b)$ and $U^*(a)\subseteq U^*(b)$. Then $f_{ab}$ is strictly monotone by Proposition~\ref{prop1}, but not upper cone preserving by Theorem~\ref{th2}.
\end{proof}

In Theorem~\ref{th2} we characterized when the mapping $f_{ab}$ is upper cone preserving. Now we show when this mapping is strictly upper cone preserving.

\begin{theorem}
Let $(P,\leq)$ be a poset and $a,b\in P$ with $a\neq b$. Then $f_{ab}$ is strictly upper cone preserving if and only if $|L(a)|\leq2$ and $U^*(a)=U(b)$.
\end{theorem}

\begin{proof}
First assume $f_{ab}$ to be strictly upper cone preserving. Then $b\leq a$ would imply
\[
a\in U(b)=U(b,b)=U(f_{ab}(a),f_{ab}(b))=f_{ab}(U(a,b))=f_{ab}(U(a))=U^*(a)\cup\{b\}
\]
and hence $a=b$, a contradiction. Therefore $b\not\leq a$. Now
\[
b\in U(b)=U(b,b)=U(f_{ab}(a),f_{ab}(b))=f_{ab}(U(a,b))=U(a,b)\subseteq U(a)
\]
since $a\notin U(a,b)$ and hence $a\leq b$, i.e.\ $a<b$ and therefore $U(b)\subseteq U^*(a)$. If $c\in U^*(a)$ then
\[
c\in U(c)=f_{ab}(U(c))=f_{ab}(U(a,c))=U(f_{ab}(a),f_{ab}(c))=U(b,c)\subseteq U(b)
\]
showing $U^*(a)\subseteq U(b)$. Altogether, we obtain $U^*(a)=U(b)$. Now $|L(a)|>2$ would imply that there exist $d,e\in P$ with $d\neq e$ and $d,e<a$ and hence
\[
f_{ab}(U(d,e))=(U(d,e))\setminus\{a\}\neq U(d,e)=U(f_{ab}(d),f_{ab}(e))
\]
contradicting the fact that $f_{ab}$ is strongly upper cone preserving. Hence $|L(a)|\leq2$. If, conversely, $|L(a)|\leq2$ and $U^*(a)=U(b)$ and $g,h\in P\setminus\{a\}$ then
\begin{align*}
f_{ab}(U(a,g)) & =\left\{
\begin{array}{ll}
f_{ab}(U(a))=U^*(a)=U(b)=U(b,g)=U(f_{ab}(a),f_{ab}(g)) & \text{if }g\leq a, \\
U(a,g)=U(b,g)=U(f_{ab}(a),f_{ab}(g))                   & \text{otherwise},
\end{array}
\right. \\
f_{ab}(U(g,h)) & =U(g,h)=U(f_{ab}(g),f_{ab}(h))\text{ if }g\neq h
\end{align*}
and hence $f_{ab}$ is strictly upper cone preserving.
\end{proof}

\section{Chains}

Chains are relatively simple posets. We derive an easy condition under which every strictly monotone mapping on a chain is upper cone preserving.

\begin{proposition}
Let $\mathbf P=(P,\leq)$ be a chain and $f$ a strictly monotone mapping from $P$ to $P$. Then $f$ is upper cone preserving if and only if $f(P)\in\Fil\mathbf P$.
\end{proposition}

\begin{proof}
First assume $f(P)\in\Fil\mathbf P$. Let $a,b\in P$ with $a\leq b$ and $c\in U(f(a),f(b))$. Then $c\geq f(b)$. Since $f(P)\in\Fil\mathbf P$, there exists some $d\in P$ with $c=f(d)$. Now $d<b$ would imply $c=f(d)<f(b)$, a contradiction. Hence $d\geq b$ and $c=f(d)\in f(U(b))=f(U(a,b))$. This shows $U(f(a),f(b))\subseteq f(U(a,b))$. The opposite inclusion follows from Lemma~\ref{lem2} (iv). Hence $f$ is upper cone preserving. The rest of the proof follows from Lemma~\ref{lem2} (ii).
\end{proof}

Another interesting question concerns posets which are semilattices. Because every semilattice homomorphism is a monotone mapping, we can ask when every monotone mapping of a given semilattice into itself is a homomorphism. Using the method developed by Berrone (\cite B), we can prove the following result.

\begin{theorem}\label{th1}
A join-semilattice $(P,\vee)$ is a chain if and only if every monotone mapping from $P$ to $P$ is a homomorphism.
\end{theorem}

\begin{proof}
Let $\mathbf P=(P,\vee)$ be a join-semilattice and $a,b\in P$. If $\mathbf P$ is a chain and $f$ a monotone mapping from $P$ to $P$ then
\[
f(a\vee b)=\left\{
\begin{array}{ll}
f(b)=f(a)\vee f(b) & \text{if }a\leq b, \\
f(a)=f(a)\vee f(b) & \text{otherwise}
\end{array}
\right.
\]
and hence $f$ is a homomorphism. Now assume $\mathbf P$ not to be a chain. Then there exist $c,d\in P$ with $c\parallel d$. Define $g:P\rightarrow P$ by
\[
g(x):=\left\{
\begin{array}{ll}
c       & \text{if }x<c\vee d, \\
c\vee d & \text{otherwise}.
\end{array}
\right.
\]
Assume $a\leq b$. If $a<c\vee d$ then $g(a)=c\leq g(b)$. If $a\not<c\vee d$ then $b\not<c\vee d$ and hence $g(a)=c\vee d=g(b)$. This shows that $g$ is monotone. But $g$ is not a homomorphism since
\[
g(c\vee d)=c\vee d\neq c=c\vee c=g(c)\vee g(d).
\]
We have proved that there exists a monotone mapping from $P$ to $P$ that is not a homomorphism .
\end{proof}

By duality, Theorem~\ref{th1} also holds for meet-semilattices and hence also for lattices. On the other hand, the result of Theorem~\ref{th1} can be extended to direct products of chains. For this, let us recall the following concepts.

For $i=1,2$ let $f_i\colon A_i\rightarrow B_i$. Then $f_1\times f_2$ denotes the mapping from $A_1\times A_2$ to $B_1\times B_2$ defined by
\[
(f_1\times f_2)(x_1,x_2):=(f_1(x_1),f_2(x_2))\text{ for all }(x_1,x_2)\in A_1\times A_2.
\]
A mapping $g:A_1\times A_2\rightarrow B_1\times B_2$ is called {\em directly decomposable} if there exist $g_1\colon A_1\rightarrow B_1$ and $g_2\colon A_2\rightarrow B_2$ with $g_1\times g_2=g$.

Let $C$ be a chain and $P:=C\times C$. As proved in \cite{CGL}, every lattice homomorphism from $P$ to $P$ is directly decomposable since $P$ is a lattice and the variety of lattices is congruence distributive. We can ask if monotone directly decomposable mappings from $P$ to $P$ are lattice homomorphisms. The following corollary of Theorem~\ref{th1} gives a positive answer.

\begin{corollary}
Let $(C_1,\leq)$, $(C_2,\leq)$ be chains, $(P,\leq):=(C_1,\leq)\times(C_2,\leq)$ and $f$ a monotone directly decomposable mapping from $P$ to $P$. Then $f$ is a lattice homomorphism.
\end{corollary}

\begin{proof}
If $f=f_1\times f_2$ with $f_i\colon C_i\rightarrow C_i$ for $i=1,2$ then $f_1,f_2$ are a monotone and, by Theorem~\ref{th1}, also (semi-)lattice homomorphisms which implies that $f$ is (semi-)lattice homomorphism, too.
\end{proof}

Direct decomposability of homomorphisms was investigated by the authors and M.~Goldstern in \cite{CGL}. For mappings which need not be homomorphisms we cannot use methods involved in congruence distributive varieties. A simple characterization of directly decomposable mappings is formulated in the following lemma.

\begin{lemma}
Let $A_1,A_2,B_1,B_2$ be non-void sets and $f:A_1\times A_2\rightarrow B_1\times B_2$ and for $i=1,2$ let $p_i$ denote the projection of $B_1\times B_2$ onto $B_i$. Then the following are equivalent:
\begin{enumerate}[{\rm(i)}]
\item $f$ is decomposable,
\item $p_1(f(x_1,x_2))=p_1(f(x_1,y_2))$ and $p_2(f(x_1,x_2))=p_2(f(y_1,x_2))$ for all $x_1,y_1\in A_1$ and $x_2,y_2\in A_2$.
\end{enumerate}
\end{lemma}

\begin{proof}
$\text{}$ \\
(i) $\Rightarrow$ (ii): \\
If $f=f_1\times f_2$ then
\begin{align*}
p_1(f(x_1,x_2)) & =p_1(f_1(x_1),f_2(x_2))=f_1(x_1)=p_1(f_1(x_1),f_2(y_2))=p_1(f(x_1,y_2)), \\
p_2(f(x_1,x_2)) & =p_2(f_1(x_1),f_2(x_2))=f_2(x_2)=p_2(f_1(y_1),f_2(x_2))=p_2(f(y_1,x_2))
\end{align*}
for all $x_1,y_1\in A_1$ and $x_2,y_2\in A_2$. \\
(ii) $\Rightarrow$ (i): \\
Let $a_1\in A_1$ and $a_2\in A_2$ and for $i=1,2$ define $f_i\colon A_i\rightarrow B_i$ by
\begin{align*}
f_1(x_1) & :=p_1(f(x_1,a_2))\text{ for all }x_1\in A_1, \\
f_2(x_2) & :=p_2(f(a_1,x_2))\text{ for all }x_2\in A_2.
\end{align*}
Because of (ii), $f_1$ and $f_2$ are well-defined and
\[
f(x_1,x_2)=(p_1(f(x_1,x_2)),p_2(f(x_1,x_2)))=(p_1(f(x_1,a_2)),p_2(f(a_1,x_2)))=(f_1(x_1),f_2(x_2))
\]
for all $(x_1,x_2)\in A_1\times A_2$, i.e.\ $f=f_1\times f_2$.
\end{proof}

Instead of join-semilattices we can investigate posets whose upper cones $U(x,y)$ have a minimal element. Of course, every join-semilattice has this property, but there many other examples of such posets, e.g.\ all finite up-directed posets.

\begin{theorem}
If $(P,\leq)$ is a poset, $a,b\in P$, $a\parallel b$ and $U(a,b)$ has a minimal element then there exists a monotone mapping $f$ from $P$ to $P$ with $f(U(a,b))\neq U(f(a),f(b))$ and hence there exists a monotone mapping from $P$ to $P$ which is not strictly upper cone preserving.
\end{theorem}

\begin{proof}
Let $(P,\leq)$ be a poset and $a,b,c\in P$ and assume $a\parallel b$ and that $c$ is a minimal element of $U(a,b)$. Define $f:P\rightarrow P$ by
\[
f(x):=\left\{
\begin{array}{ll}
a & \text{if }x<c, \\
c & \text{otherwise}.
\end{array}
\right.
\]
Let $d,e\in P$ with $d\leq e$. If $d<c$ then $f(d)=a\leq f(e)$. If $d\not<c$ then $e\not<c$ and hence $f(d)=c=f(e)$. This shows that $f$ is monotone. We have $a,b\leq c$. Since $a=c$ would imply $b\leq c=a$ and $b=c$ would imply $a\leq c=b$, we have $a,b<c$ and therefore $a\in U(a)=U(a,a)=U(f(a),f(b))$. Now assume $f(U(a,b))=U(f(a),f(b))$. Then $a\in f(U(a,b))$ and hence there exists some $d\in U(a,b)$ with $f(d)=a$. Since $c$ is a minimal element of $U(a,b)$ we have $d\not<c$ and hence $a=f(d)=c$, a contradiction. Therefore $f(U(a,b))\neq U(f(a),f(b))$.
\end{proof}

\begin{corollary}
If $(P,\leq)$ is a poset which is not a chain and which satisfies the Descending Chain Condition then there exists a monotone mapping from $P$ to $P$ which is not strictly upper cone preserving and hence not upper cone preserving.
\end{corollary}

On the other hand, if a poset in question is a chain, we can give a necessary and sufficient condition for a monotone mapping to be upper cone preserving.

\begin{proposition}
Let $(C,\leq)$ be a chain and $f\colon C\rightarrow C$ monotone. Then $f$ is upper cone preserving if and only if $U(f(x))\subseteq f(C)$ for all $x\in C$.
\end{proposition}

\begin{proof}
Let $a\in C$. If $f$ is upper cone preserving then
\[
U(f(a))=U(f(a),f(a))=f(U(a,a))\subseteq f(C).
\]
Conversely, assume $U(f(x))\subseteq f(C)$ for all $x\in C$. Since $f$ is monotone, we have $f(U(a))\subseteq U(f(a))$ according to Lemma~\ref{lem2} (i). Now let $b\in U(f(a))$. If $b=f(a)$ then $b\in f(U(a))$. Now assume $b>f(a)$. Since $b\in U(f(a))\subseteq f(C)$, there exists some $c\in C$ with $f(c)=b$. Now $c\leq a$ would imply $b=f(c)\leq f(a)$, a contradiction. Hence $c\in U(a)$ and therefore $b=f(c)\in f(U(a))$. This shows $U(f(a))\subseteq f(U(a))$ and hence $f(U(a))=U(f(a))$. Now, for $x,y\in C$ we have
\[
U(f(x),f(y))=\left\{
\begin{array}{ll}
U(f(y))=f(U(y))=f(U(x,y)) & \text{if }x\leq y, \\
U(f(x))=f(U(x))=f(U(x,y)) & \text{otherwise},
\end{array}
\right.
\]
i.e.\ $f$ is upper cone preserving.
\end{proof}

\section{Ordinal sums and equivalence relations}

We have seen that the Descending Chain Condition together with the property that every monotone mapping is strictly upper cone preserving forces a poset to be a chain. It seems that our conditions are too restrictive. In fact, if we replace monotone mappings by strictly monotone ones, we can obtain a richer structure of posets in which strictly monotone mappings are strongly upper cone preserving.

The {\em ordinal sum} of two posets $(A,\leq)$ and $(B,\leq)$ with $A\cap B=\emptyset$ is the poset with base set $A\cup B$ where the order inside $A$ and inside $B$ coincides with the original one and $A<B$, i.e.\ every element of $A$ is below every element of $B$. Now, we can state the following result.

\begin{proposition}
Every strictly monotone mapping on the ordinal sum of an antichain and a finite chain is strongly upper cone preserving.
\end{proposition}

\begin{proof}
If $f$ is a strictly monotone mapping on the ordinal sum $(P,\leq)$ of an antichain $(A,\leq)$ and a finite chain $(C,\leq)$, $a,b\in P$ and $f(a)\neq f(b)$ then $f(A)\subseteq A$, $f(x)=x$ for all $x\in C$ and
\[
f(U(a,b))=\left\{
\begin{array}{ll}
f(C)=C=U(f(a),f(b))                 & \text{if }a,b\in A, \\
f(U(b))=U(b)=U(f(a),b)=U(f(a),f(b)) & \text{if }a\in A\text{ and }b\in C, \\
U(a,b)=U(f(a),f(b))                 & \text{if }a,b\in C.
\end{array}
\right.
\]
\end{proof}

\begin{example}
Examples of such ordinal sums are visualized in the Figure~2:

\vspace*{1mm}

\begin{center}
\setlength{\unitlength}{7mm}
\begin{picture}(16,8)
\put(0,0){\circle*{.3}}
\put(2,0){\circle*{.3}}
\put(5,0){\circle*{.3}}
\put(7,0){\circle*{.3}}
\put(9.6,0){\ldots}
\put(11,0){\circle*{.3}}
\put(11.6,0){\ldots}
\put(13,0){\circle*{.3}}
\put(13.6,0){\ldots}
\put(15,0){\circle*{.3}}
\put(15.6,0){\ldots}
\put(1,2){\circle*{.3}}
\put(6,2){\circle*{.3}}
\put(6,4){\circle*{.3}}
\put(13,2){\circle*{.3}}
\put(13,4){\circle*{.3}}
\put(13,6){\circle*{.3}}
\put(13,8){\circle*{.3}}
\put(1,2){\line(-1,-2)1}
\put(1,2){\line(1,-2)1}
\put(6,2){\line(-1,-2)1}
\put(6,2){\line(1,-2)1}
\put(6,2){\line(0,1)2}
\put(13,2){\line(-1,-1)2}
\put(13,2){\line(1,-1)2}
\put(13,0){\line(0,1)8}
\put(.6,-1){{\rm(a)}}
\put(5.6,-1){{\rm(b)}}
\put(12.6,-1){{\rm(c)}}
\put(7.3,-2){{\rm Fig.\ 2}}
\end{picture}
\end{center}

\vspace*{8mm}

\end{example}

Every mapping $f\colon A\rightarrow B$ induces an equivalence relation $\Theta$ on $A$ by defining $(x,y)\in\Theta$ if $f(x)=f(y)$. This equivalence relation is called the {\em kernel} of $f$, usually denoted by $\ker f$. The question when for a given poset $(P,\leq)$ and a given mapping $F\colon P\rightarrow P$ the quotient set $P/(\ker f)$ is again a poset is answered in the next theorem.

Let $(P,\leq)$ and $(Q,\leq)$ be posets and $f\colon P\rightarrow Q$. Recall that $f$ is called {\em strongly monotone} if it is monotone and $a,b\in P$ and $f(a)\leq f(b)$ imply that there exist $a',b'\in P$ with $f(a')=f(a)$, $f(b')=f(b)$ and $a'\leq b'$.

\begin{definition}\label{def1}
Le $\mathbf P=(P,\leq)$ be a poset. An equivalence relation $\Theta$ on $P$ is called an {\em $S$-equivalence} on $\mathbf P$ if it satisfies the following two conditions for all $a,a',b,b'\in P$:
\begin{enumerate}[{\rm(i)}]
\item If $a,b,b',c\in P$, $a\leq b$, $b'\leq c$ and $(b,b')\in\Theta$ then there exist $a'\in[a]\Theta$ and $c'\in[c]\Theta$ with $a'\leq c'$,
\item if $a,a',b,b'\in P$, $a\leq b$, $b'\leq a'$ and $(a,a'),(b,b')\in\Theta$ then $(a,b)\in\Theta$.
\end{enumerate}
\end{definition}

\begin{theorem}
Let $\mathbf P=(P,\leq)$ be a poset, $f\colon P\rightarrow P$ strongly monotone and $\Theta$ an $S$-equivalence on $\mathbf P$ and define $[a]\Theta\leq[b]\Theta$ if there exist $a'\in[a]\Theta$ and $b'\in[b]\Theta$ with $a'\leq b'$. Then
\begin{enumerate}[{\rm(i)}]
\item $\ker f$ is an $S$-equivalence on $\mathbf P$,
\item $(P/\Theta,\leq)$ is a poset and $x\mapsto[x]\Theta$ strongly monotone.
\end{enumerate}
\end{theorem}

\begin{proof}
\
\begin{enumerate}[(i)]
\item Put $\Phi:=\ker f$ and assume $a,b,b',c\in P$, $a\leq b$, $b'\leq c$ and $(b,b')\in\Phi$. Then
\[
f(a)\leq f(b)=f(b')\leq f(c).
\]
Since $f$ is strongly monotone there exist $a',c'\in P$ with $f(a')=f(a)$, $f(c')=f(c)$ and $a'\leq c'$. Hence $a'\in[a]\Phi$, $c'\in[c]\Phi$ and $a'\leq c'$ proving (i) of Definition~\ref{def1}. Next assume $a,a',b,b'\in P$, $a\leq b$, $b'\leq a'$ and $(a,a'),(b,b')\in\Phi$. Then
\[
f(a)\leq f(b)=f(b')\leq f(a')=f(a)
\]
and hence $f(a)=f(b)$, i.e.\ $(a,b)\in\Phi$ proving (ii) of Definition~\ref{def1}.
\item We consider the binary relation $\leq$ on $P/\Theta$. Obviously, $\leq$ is reflexive. Assume $a,b\in P$, $[a]\Theta\leq[b]\Theta$ and $[b]\Theta\leq[a]\Theta$. Then there exist $a',a''\in[a]\Theta$ and $b',b''\in[b]\Theta$ with $a'\leq b'$ and $b''\leq a''$. Since $(a',a''),(b',b'')\in\Theta$, we conclude by (ii) of Definition~\ref{def1} that $(a',b')\in\Theta$. This shows $[a]\Theta=[a']\Theta=[b']\Theta=[b]\Theta$ proving antisymmetry of $\leq$. Now assume $a,b,c\in P$, $[a]\Theta\leq[b]\Theta$ and $[b]\Theta\leq[c]\Theta$. Then there exist $a'\in[a]\Theta$, $b',b''\in[b]\Theta$ and $c'\in[c]\Theta$ with $a'\leq b'$ and $b''\leq c'$. Since $(b',b'')\in\Theta$, we conclude by (i) of Definition~\ref{def1} that there exist $a''\in[a']\Theta$ and $c''\in[c']\Theta$ with $a''\leq c''$. Now $a''\in[a]\Theta$ and $c''\in[c]\Theta$ which shows $[a]\Theta\leq[c]\Theta$ proving transitivity of $\leq$. Altogether, $(P/\Theta,\leq)$ is a poset. Clearly, $x\mapsto[x]\Theta$ is monotone and by the definition of $\leq$ on $P/\Theta$, this mapping is strongly monotone.
\end{enumerate}
\end{proof}

\begin{example}
Consider the poset $\mathbf P=(P,\leq)$ visualized in Figure~3:

\vspace*{-2mm}

\begin{center}
\setlength{\unitlength}{7mm}
\begin{picture}(4,8)
\put(1,3){\circle*{.3}}
\put(3,3){\circle*{.3}}
\put(1,5){\circle*{.3}}
\put(3,5){\circle*{.3}}
\put(2,7){\circle*{.3}}
\put(2,1){\circle*{.3}}
\put(1,3){\line(0,1)2}
\put(1,3){\line(1,1)2}
\put(3,3){\line(-1,1)2}
\put(3,3){\line(0,1)2}
\put(2,7){\line(-1,-2)1}
\put(2,7){\line(1,-2)1}
\put(2,1){\line(-1,2)1}
\put(2,1){\line(1,2)1}
\put(.3,2.85){$a$}
\put(3.4,2.85){$b$}
\put(.3,4.85){$c$}
\put(3.4,4.85){$d$}
\put(1.85,7.4){$1$}
\put(1.85,.25){$0$}
\put(1.2,-.75){{\rm Fig.\ 3}}
\end{picture}
\end{center}

\vspace*{4mm}

Let $f\colon P\rightarrow P$ be defined by
\[
\begin{array}{c|cccccc}
  x  & 0 & a & b & c & d & 1 \\
\hline
f(x) & 0 & a & a & c & c & 1
\end{array}
\]
and put $\Theta:=\ker f$. Then $f$ is strongly monotone, $\Theta=\{0\}^2\cup\{a,b\}^2\cup\{c,d\}^2\cup\{1\}^2$ is an $S$-equivalence on $\mathbf P$ and $(P/\Theta,\leq)=(\{[0]\Theta,[a]\Theta,[c]\Theta,[1]\Theta\},\leq)$ is again a poset where $[0]\Theta<[a]\Theta<[c]\Theta<[1]\Theta$.
\end{example}

Authors' addresses:

Ivan Chajda \\
Palack\'y University Olomouc \\
Faculty of Science \\
Department of Algebra and Geometry \\
17.\ listopadu 12 \\
771 46 Olomouc \\
Czech Republic \\
ivan.chajda@upol.cz

Helmut L\"anger \\
TU Wien \\
Faculty of Mathematics and Geoinformation \\
Institute of Discrete Mathematics and Geometry \\
Wiedner Hauptstra\ss e 8-10 \\
1040 Vienna \\
Austria, and \\
Palack\'y University Olomouc \\
Faculty of Science \\
Department of Algebra and Geometry \\
17.\ listopadu 12 \\
771 46 Olomouc \\
Czech Republic \\
helmut.laenger@tuwien.ac.at

\end{document}